\documentclass{amsart}

\usepackage{epsf,epsfig,amsfonts,amsgen,amsmath,amstext,amsbsy,amsopn,amsthm,cases,listings,color
}
\usepackage{ebezier,eepic}
\usepackage{color}
\usepackage{multirow}
\usepackage{epstopdf}
\usepackage{graphicx}
\usepackage{pgf,tikz}
\usepackage{mathrsfs}
\usepackage[marginal]{footmisc}
\usepackage{enumitem}
\usepackage{appendix}
\usepackage{graphicx}
\usepackage{ amssymb,amsmath,amscd, enumitem,
	latexsym,amsthm, verbatim}
\usepackage{amssymb,latexsym,eufrak,amsmath,amscd}
\usepackage[all]{xy}
\usepackage{pdfsync}
\usepackage{url}

\usepackage{pgf,tikz,pgfplots}
\usepackage{subfigure}
\usepackage{tikz-cd}
\usepackage{mathrsfs}
\usetikzlibrary{arrows}

\newtheorem{theorem}{Theorem}[section]

\newtheorem{proposition}[theorem]{Proposition}

\theoremstyle{definition}
\newtheorem{definition}[theorem]{Definition}

\theoremstyle{remark}
\newtheorem{remark}[theorem]{Remark}
\newtheorem{claim}[theorem]{Claim}

\numberwithin{equation}{section}

\newcommand{\Hom}{\textnormal{Hom}}

\newcommand{\bp}{{\mathbb{P}^r}}

\newtheorem{thmx}{\textnormal{\textbf{Theorem}}}

\begin{document}

\title{A remark on the Castelnuovo-Mumford regularity of powers of ideal sheaves}

%    Remove any unused author tags.

%    author one information
\author{}
\address{}
\curraddr{}
\email{}
\thanks{}

%    author two information
\author{Shijie Shang}
\address{Department of Mathematics, Statistics, and Computer Science, University of Illinois at Chicago, Chicago, IL 60607}
\curraddr{}
\email{sshang8@uic.edu}
\thanks{}

\subjclass[2010]{Primary: 14F17, 14M10. Secondary: 14M06.}

\keywords{Castelnuovo-Mumford regularity, ideal sheaves, complete intersections}

\date{\today}

\dedicatory{}

\begin{abstract}
We show that a bound of the Castelnuovo-Mumford regularity of any power of the ideal sheaf of a smooth projective complex variety $X\subseteq\mathbb{P}^r$ is sharp exactly for complete intersections, provided the variety $X$ is cut out scheme-theoretically by several hypersurfaces in $\mathbb{P}^r$. This generalizes a result of Bertram-Ein-Lazarsfeld.
\end{abstract}

\maketitle

%%%%%%%%%%%%%%%%%%%%%%%%%%%%%%%%%%%%%%%%%%%%%%%%%
\section{Introduction}

Let $X\subseteq \mathbb{P}^r$ be a smooth complex projective (possibly disconnected) variety of (pure) dimension $n$ and codimension $e=r-n$ and $\mathcal{I}_X$ be the ideal sheaf of $X$ in $\mathbb{P}^r$. Bertram, Ein and Lazarsfeld \cite{BEL} prove a theorem which asserts the vanishing of the higher cohomology of suitable twists of the $a$-th power of the ideal sheaf $\mathcal{I}_X$.
\begin{theorem}\label{thm1}(\cite[Proposition~1]{BEL})
	Let $X\subseteq \mathbb{P}^r$ be a smooth variety of dimension $n$ and codimension $e$ cut out scheme-theoretically by hypersurfaces of degrees $d_1\geq d_2\geq\dotsb\geq d_m$. For all integers $i\geq 1$ and all  integers $a\geq 1$, we have $H^i(\mathbb{P}^r,\mathcal{I}^a_X(k))=0$ provided $k\geq ad_1+d_2+d_3+\dotsb+d_e-r$.
\end{theorem} 

A coherent sheaf $\mathcal{F}$ on a projective space $\bp$ is \textit{$k$-regular} if $H^i(\bp,\mathcal{F}(k-i))=0$ for all $i>0$, and a projective variety $X\subseteq \bp$ is \textit{$k$-regular} if the ideal sheaf $\mathcal{I}_X$ is $k$-regular. As an application of Theorem \ref{thm1}, Bertram, Ein and Lazarsfeld \cite{BEL} yield a bound of the Castelnuovo-Mumford regularity of $X$ when $X$ is smooth and show that complete intersections are the only borderline cases.
\begin{theorem}\label{thm2}(\cite[Corollary~4]{BEL})
	Let $X\subseteq \mathbb{P}^r$ be a smooth variety of dimension $n$ and codimension $e$ cut out scheme-theoretically by hypersurfaces of degrees $d_1\geq d_2\geq\dotsb\geq d_m$. \\
	(i) The variety $X$ is $(d_1+d_2+\dotsb+d_e-e+1)$-regular, and\\
	(ii) The variety $X$ fails to be $(d_1+d_2+\dotsb+d_e-e)$-regular if and only if $X$ is the complete intersection of hypersurfaces of degrees $d_1,d_2,\dotsc,d_e$.
\end{theorem}

There has been interest in recent years in generalizing the vanishing theorem and establishing bounds of Castelnuovo-Mumford regularity of projective schemes in terms of their defining equations. Chardin-Ulrich \cite{CB}, de Fernex-Ein \cite{dFE} and Niu \cite{Niu,Niu2} generalize Theorem \ref{thm2} allowing $X$ to be a singular variety or, more generally, a singular scheme in $\mathbb{P}^r$. Promoting the results of \cite{BEL} to statements about arithmetic regularity, Ein, Hà and Lazarsfeld \cite{EHL} prove some saturation bounds for the ideals of smooth complex projective varieties and their powers.

The purpose of this note is to generalize Theorem \ref{thm2}. We prove the following.
\begin{thmx}\label{mainthm}
	Let $X\subseteq \mathbb{P}^r$ be a smooth variety of dimension $n$ and codimension $e$ cut out scheme-theoretically by hypersurfaces of degrees $d_1\geq d_2\geq\dotsb\geq d_m$. For any integer $a\geq 1$, we have\\
	(i) The sheaf $\mathcal{I}^a_X$ is $(ad_1+d_2+d_3+\dotsb+d_e-e+1)$-regular, and\\
	(ii) The sheaf $\mathcal{I}^a_X$ fails to be $(ad_1+d_2+d_3+\dotsb+d_e-e)$-regular if and only if $X$ is the complete intersection of hypersurfaces of degrees $d_1,d_2,\dotsc,d_e$.
\end{thmx}

\section{Preliminaries}
In this section we review some basic facts and definitions. Throughout this note we work over $\mathbb{C}$.

Any arithmetically Gorenstein subscheme $X\subseteq\mathbb{P}^r$ is arithmetically Cohen-Macaulay. Note that any complete intersection subscheme in a projective space is arithmetically Gorenstein.
\begin{definition}(\cite[Definition~1.3.1]{MR})
	We say that two equidimensional subschemes $V_1$ and $V_2$ of $\mathbb{P}^r$ of the same dimension are \textit{algebraically directly G-linked} by an arithmetically Gorenstein subscheme $X\subseteq\mathbb{P}^r$ if $I(X)\subseteq I(V_1)\cap I(V_2)$ and we have $[I(X):I(V_1)]=I(V_2)$ and $[I(X):I(V_2)]=I(V_1)$. The scheme $V_1$ is said to be \textit{residual} to $V_2$ in $X$.
\end{definition}
\begin{proposition}\label{cor}
	Let $X\subseteq\mathbb{P}^r$ be an arithmetically Gorenstein scheme of dimension $n$ and $V_1$ be a closed Cohen-Macaulay subscheme of dimension $n$ in $X$. There exists an exact sequence
	$$
	0\rightarrow\omega_{V_1}\otimes\omega_X^{-1}\rightarrow\mathcal{O}_X\rightarrow\mathcal{O}_{V_2}\rightarrow0,
	$$
	where $V_2$ is a closed Cohen-Macaulay subscheme of dimension $n$ and residual to $V_1$ in $X$.
\end{proposition}
\begin{proof}
	This statement follows from  \cite[Theorem~21.23]{Eisenbud} and  \cite[Proposition~5.2.6]{Migliore} immediately. Here $\mathcal{O}_X\rightarrow\mathcal{O}_{V_2}$ is the natural restriction morphism.
\end{proof}
\section{Proof of Theorem \ref{mainthm}}
This section is devoted to the proof of Theorem \ref{mainthm}. We start by fixing notations and set-up. As in the introduction, $X\subseteq \mathbb{P}^r$ is a smooth variety of dimension $n$ and codimension $e$ cut out scheme-theoretically by hypersurfaces of degrees $d_1\geq d_2\geq\dotsb\geq d_m>0$. Set $\Delta_1=d_1+d_2+\dotsb+d_e-e-1$ and $\Delta_a=ad_1+d_2+d_3+\dotsb+d_e-e-1$ for any integer $a\geq 2$. By the assumption on the variety $X$, we have the following exact sequence
\begin{equation}\label{3.1}
\bigoplus^m_{i=1}\mathcal{O}_{\bp}(-d_i)\rightarrow\mathcal{I}_X\rightarrow0.
\end{equation}
\subsection{Proof of Theorem \ref{mainthm} (i)} 
Theorem \ref{mainthm} (i) is an immediate consequence of Theorem \ref{thm1}. For completeness, we include a proof here. 

We prove Theorem \ref{mainthm} (i) by induction on $a$. When $a=1$, the statement follows from Theorem \ref{thm2} (i). Suppose that the statement holds for $a-1$ case, where $a\geq 2$. Now we consider the following exact sequence
\begin{equation}\label{3.2}
0\rightarrow\mathcal{I}^a_X\rightarrow\mathcal{I}^{a-1}_X\rightarrow \mathcal{I}^{a-1}_X/\mathcal{I}^a_X\rightarrow0,
\end{equation}
where $\mathcal{I}^{a-1}_X/\mathcal{I}^a_X\cong\textnormal{Sym}^{a-1}\mathcal{N}^*_{X/\bp}$ and $\mathcal{N}^*_{X/\bp}$ is the conormal sheaf of $X$ in $\bp$.
It suffices to show that $H^i(\bp,\mathcal{I}^a_X(\Delta_a+2-i))=0$ for any integer $i\geq 1$.

First, for any integer $i$ such that $1\leq i\leq n+1$, we have
\[
\begin{split}
\Delta_a+2-i&\geq ad_1+d_2+d_3+\dotsb+d_e-e+1-(n+1)\\
&=ad_1+d_2+d_3+\dotsb+d_e-r
\end{split}
\]
By Theorem \ref{thm1}, we see that $H^i(\bp,\mathcal{I}^a_X(\Delta_a+2-i))=0$. 

For any integer $i\geq n+2$, we consider the following exact sequence deduced from (\ref{3.2}),
\begin{equation}\label{3.3}
H^{i-1}(\bp,(\mathcal{I}^{a-1}_X/\mathcal{I}^a_X)(\Delta_a+2-i))\rightarrow H^i(\bp,\mathcal{I}^a_X(\Delta_a+2-i))
\rightarrow H^i(\bp,\mathcal{I}^{a-1}_X(\Delta_a+2-i)).
\end{equation}
For any integer $i\geq n+2$, we have $i-1\geq n+1>\dim X$. Thus, we obtain
$$
H^{i-1}(\bp,(\mathcal{I}^{a-1}_X/\mathcal{I}^a_X)(\Delta_a+2-i))\cong H^{i-1}(\bp,\textnormal{Sym}^{a-1}\mathcal{N}^*_{X/\bp}(\Delta_a+2-i))=0
$$
The induction hypothesis implies $H^i(\bp,\mathcal{I}^{a-1}_X(\Delta_a+2-i))=0$ for any integer $i\geq n+2$. Therefore, it follows from (\ref{3.3}) that $H^i(\bp,\mathcal{I}^a_X(\Delta_a+2-i))=0$ for any integer $i\geq n+2$. This completes the proof of Theorem \ref{mainthm} (i).
\subsection{Proof of Theorem \ref{mainthm} (ii)} When $e=1$, the statement follows from the fact that $\mathcal{I}_X\cong\mathcal{O}_{\bp}(-\deg X)$ immediately. In the following proof, we assume that $e\geq 2$.

$\Leftarrow\,:$ Since $X$ is the complete intersection of hypersurfaces of degrees $d_1,d_2,\dotsc,d_e$, then we deduce that 
$
\mathcal{N}^*_{X/\bp}\cong\bigoplus^{e}_{i=1}\mathcal{O}_X(-d_i).
$
Thus, we obtain
$$
\textnormal{Sym}^{a-1}\mathcal{N}^*_{X/\bp}\cong\mathcal{O}_X(-(a-1)d_1)\oplus \mathcal{E},
$$
where $\mathcal{E}$ is a direct sum of line bundles in the form of $\mathcal{O}_X(m)$ where we have $m\geq -(a-1)d_1$. It suffices to show that
$
H^{n+1}(\bp,\mathcal{I}^a_X(\Delta_a-n))\neq0.
$
We obtain the following exact sequence deduced from (\ref{3.2}),
$$
H^n(\bp,\mathcal{I}^{a-1}_X(\Delta_a-n))\rightarrow H^n(\bp,(\mathcal{I}^{a-1}_X/\mathcal{I}^a_X)(\Delta_a-n))
$$
$$
\rightarrow H^{n+1}(\bp,\mathcal{I}^{a}_X(\Delta_a-n))
\rightarrow H^{n+1}(\bp,\mathcal{I}^{a-1}_X(\Delta_a-n)),
$$
where we have $H^n(\bp,\mathcal{I}^{a-1}_X(\Delta_a-n))=0$ and $H^{n+1}(\bp,\mathcal{I}^{a-1}_X(\Delta_a-n))=0$ by Theorem \ref{thm1}. Thus, we have
$$
H^n(\bp,(\mathcal{I}^{a-1}_X/\mathcal{I}^a_X)(\Delta_a-n))\cong H^{n+1}(\bp,\mathcal{I}^{a}_X(\Delta_a-n)).
$$
From $\mathcal{I}^{a-1}_X/\mathcal{I}^a_X\cong\textnormal{Sym}^{a-1}\mathcal{N}^*_{X/\bp}\cong\mathcal{O}_X(-(a-1)d_1)\oplus \mathcal{E}$, we see that $\mathcal{O}_X(\Delta_1-n)$ is a direct summand of $(\mathcal{I}^{a-1}_X/\mathcal{I}^a_X)(\Delta_a-n)$. By Serre duality, we have
$$
H^n(X,\mathcal{O}_X(\Delta_1-n))^*\cong H^0(X,\mathcal{O}_X)\neq 0.
$$
Thus, we conclude that $H^n(\bp,(\mathcal{I}^{a-1}_X/\mathcal{I}^a_X)(\Delta_a-n))\neq 0$, which implies that 
$
H^{n+1}(\bp,\mathcal{I}^a_X(\Delta_a-n))\neq0.
$

$\Rightarrow\,:$ We consider $\dim X\geq 1$ case first. For any integer $i$ such that $1\leq i\leq n=\dim X$, we obtain $H^i(\bp,\mathcal{I}^a_X(\Delta_a+1-i))=0$ by Theorem \ref{thm1}. 

For any integer $i\geq n+2=\dim X+2$,  $H^{i-1}(\bp,(\mathcal{I}^{a-1}_X/\mathcal{I}^a_X)(\Delta_a+1-i))=0$ and $H^{i}(\bp,(\mathcal{I}^{a-1}_X/\mathcal{I}^a_X)(\Delta_a+1-i))=0$ follow from $\mathcal{I}^{a-1}_X/\mathcal{I}^a_X\cong\textnormal{Sym}^{a-1}\mathcal{N}^*_{X/\bp}$. In addition, we have the following exact sequence deduced from (\ref{3.2}),
$$
H^{i-1}(\bp,(\mathcal{I}^{a-1}_X/\mathcal{I}^a_X)(\Delta_a+1-i))\rightarrow H^i(\bp,\mathcal{I}^a_X(\Delta_a+1-i))
$$
$$
\rightarrow H^{i}(\bp,\mathcal{I}^{a-1}_X(\Delta_a+1-i))\rightarrow H^{i}(\bp,(\mathcal{I}^{a-1}_X/\mathcal{I}^a_X)(\Delta_a+1-i)).
$$
It follows from Theorem \ref{mainthm} (i) that  $H^{i}(\bp,\mathcal{I}^{a-1}_X(\Delta_a+1-i))=0$. Therefore, we have $H^i(\bp,\mathcal{I}^a_X(\Delta_a+1-i))=0$ for any integer $i\geq n+2$. Since $\mathcal{I}^a_X$ fails to be $(\Delta_a+1)$-regular, then we deduce that
$
H^{n+1}(\bp,\mathcal{I}^a_X(\Delta_a-n))\neq0.
$

Now we consider the following exact sequence deduced from (\ref{3.2}),
$$
H^n(\bp,\mathcal{I}^{a-1}_X(\Delta_a-n))\rightarrow H^n(\bp,\mathcal{I}^{a-1}_X/\mathcal{I}^a_X(\Delta_a-n))
$$
$$
\rightarrow H^{n+1}(\bp,\mathcal{I}^{a}_X(\Delta_a-n))\rightarrow H^{n+1}(\bp,\mathcal{I}^{a-1}_X(\Delta_a-n)),
$$
where we have $H^n(\bp,\mathcal{I}^{a-1}_X(\Delta_a-n))=0$ and $H^{n+1}(\bp,\mathcal{I}^{a-1}_X(\Delta_a-n))=0$ by Theorem \ref{thm1}. Thus, we obtain
$$
H^n(\bp,(\mathcal{I}^{a-1}_X/\mathcal{I}^a_X)(\Delta_a-n))\cong H^{n+1}(\bp,\mathcal{I}^{a}_X(\Delta_a-n))\neq 0.
$$
From $\mathcal{I}^{a-1}_X/\mathcal{I}^a_X\cong\textnormal{Sym}^{a-1}\mathcal{N}^*_{X/\bp}$, we have
$
H^n(X,\textnormal{Sym}^{a-1}\mathcal{N}^*_{X/\bp}(\Delta_a-n))\neq 0.
$
By Serre duality, we see that
$$
\Hom(\textnormal{Sym}^{a-1}\mathcal{N}^*_{X/\bp}(\Delta_a-n),\omega_X)
\cong H^n(X,\textnormal{Sym}^{a-1}\mathcal{N}^*_{X/\bp}(\Delta_a-n))^*\neq0,
$$
where $\omega_X$ is the dualizing sheaf of $X$.

Arguing as in the proof of Theorem 1.2 in \cite{BEL}, we can construct a complete intersection $Y\subset \bp$ cut out scheme-theoretically by $e$ hypersurfaces of degree $d_1,d_2,\dotsc,d_e$ such that $X$ is a subvariety of $Y$. Briefly speaking, we can choose general element $s_i\in H^0(\bp,\mathcal{I}_X(d_i))$ for all $1\leq i\leq e$ defining hypersurfaces $D_i\supset X$ such that the $s_i$ generate $\mathcal{I}_X$ away from a codimension 1 subset. Then we have
$$
\Hom(\textnormal{Sym}^{a-1}\mathcal{N}^*_{X/\bp}((a-1)d_1),\omega_X\otimes\omega^{-1}_Y)\\
\cong\Hom(\textnormal{Sym}^{a-1}\mathcal{N}^*_{X/\bp}(\Delta_a-n),\omega_X)\neq0
$$
due to $\omega_Y\cong\mathcal{O}_Y(\Delta_1-n)$.
Therefore, there exists a nonzero sheaf morphism $$
\psi:\textnormal{Sym}^{a-1}\mathcal{N}^*_{X/\bp}((a-1)d_1)\rightarrow\omega_X\otimes\omega^{-1}_Y.
$$
Meanwhile, it follows from (\ref{3.1}) that $\mathcal{I}_X(d_1)$ is globally generated. So $$\textnormal{Sym}^{a-1}\mathcal{N}^*_{X/\bp}((a-1)d_1)$$ is also globally generated. Hence, 
$$
H^0(\psi): H^0(X,\textnormal{Sym}^{a-1}\mathcal{N}^*_{X/\bp}((a-1)d_1))\rightarrow H^0(X,\omega_X\otimes\omega^{-1}_Y)
$$ 
is nonzero. In particular, we have $H^0(X,\omega_X\otimes\omega^{-1}_Y)\neq 0$.

Let $X'$ be the scheme which is residual to $X$ in $Y$, then we have the following exact sequence by Proposition \ref{cor}
$$
0\rightarrow\omega_X\otimes\omega^{-1}_Y\rightarrow\mathcal{O}_Y\rightarrow\mathcal{O}_{X'}\rightarrow0.
$$
By taking global sections, we have the following exact sequence
$$
0\rightarrow H^0(X,\omega_X\otimes\omega^{-1}_Y)\rightarrow H^0(Y,\mathcal{O}_Y)\rightarrow H^0(X',\mathcal{O}_{X'}).
$$
From $ H^0(X,\omega_X\otimes\omega^{-1}_Y)\neq 0$ and $H^0(Y,\mathcal{O}_Y)\cong\mathbb{C}$, we see that $$H^0(Y,\mathcal{O}_Y)\rightarrow H^0(X',\mathcal{O}_{X'})$$ is zero map. However, $H^0(Y,\mathcal{O}_Y)\rightarrow H^0(X',\mathcal{O}_{X'})$ is the restriction map, which is nontrivial if $X'\neq\emptyset$. Thus, we conclude that $X'=\emptyset$ and $Y=X$, i.e., $X$ is the complete intersection of hypersurfaces of degrees $d_1,d_2,\dotsc,d_e$.

Now we prove $\dim X=0$ case by induction on $a$. For completeness, we include a proof of $a=1$ case here. Twisting the exact sequence 
$$0
\rightarrow\mathcal{I}_X\rightarrow\mathcal{O}_\bp\rightarrow\mathcal{O}_X\rightarrow0
$$ 
by $\mathcal{O}_\bp(\Delta_1+1-i)$ and passing to cohomology, we obtain the following exact sequence
$$
H^{i-1}(\bp,\mathcal{O}_X(\Delta_1+1-i))\rightarrow H^i(\bp,\mathcal{I}_X(\Delta_1+1-i))\rightarrow H^i(\bp,\mathcal{O}_\bp(\Delta_1+1-i)).
$$
For any integer $i\geq 2$, we see that $H^{i-1}(\bp,\mathcal{O}_X(\Delta_1+1-i))=0$ due to the assumption $\dim X=0$. Meanwhile, it is clear that $H^i(\bp,\mathcal{O}_\bp(\Delta_1+1-i))=0$ for any integer $i\geq 2$. Thus, we deduce that $H^i(\bp,\mathcal{I}_X(\Delta_1+1-i))=0$ for any integer $i\geq 2$. Since $X$ is not $(\Delta_1+1)$-regular, 
then we have $$H^1(\bp,\mathcal{I}_X(\Delta_1))\neq 0.$$

Similarly as the proof of positive dimensional case, there exists a complete intersection $Y\subset \bp$ cut out scheme-theoretically by $e$ hypersurfaces of degree $d_1,d_2,\dotsc,d_e$ such that $X$ is a subvariety of $Y$. Hence, we have the following Koszul complex
\begin{equation}\label{0.2}
0\rightarrow\mathcal{O}_{\bp}(-\sum_{i=1}^{e}d_i)\rightarrow\cdots\rightarrow\bigoplus^e_{i=1}\mathcal{O}_{\bp}(-d_i)\rightarrow\mathcal{I}_Y\rightarrow0.
\end{equation}
Twisting (\ref{0.2}) by $\mathcal{O}_\bp(\Delta_1)$ and splitting it to short exact sequences, we have
$$
H^1(\bp,\mathcal{I}_Y(\Delta_1))\cong H^e(\bp,\mathcal{O}_{\bp}(-\sum_{i=1}^{e}d_i+\Delta_1))\cong H^0(\bp,\mathcal{O}_{\bp})^*\cong\mathbb{C},$$ 
which implies $h^1(\bp,\mathcal{I}_Y(\Delta_1))=1$.

Assume that $X$ is a proper subscheme of $Y$, then we can choose a zero-dimensional scheme $W$ of length $h^0(Y,\mathcal{O}_Y)-1$ and $X\subseteq W\subseteq Y$. It follows from the short exact sequence
$$
0\rightarrow\mathcal{I}_W\rightarrow\mathcal{I}_X\rightarrow\mathcal{I}_X/\mathcal{I}_W\rightarrow0
$$
and $\dim\textnormal{Supp\,}\mathcal{I}_X/\mathcal{I}_W=0$ that $H^1(\bp,\mathcal{I}_W(\Delta_1))\neq 0$.

Twisting the exact sequence $$
0\rightarrow\mathcal{I}_Y\rightarrow\mathcal{I}_W\rightarrow\mathcal{I}_W/\mathcal{I}_Y\rightarrow0
$$ by $\mathcal{O}_\bp(\Delta_1)$
and then passing to cohomology, we see
$$
h^0(\bp,\mathcal{I}_W(\Delta_1))-h^1(\bp,\mathcal{I}_W(\Delta_1))=h^0(\bp,\mathcal{I}_Y(\Delta_1))-h^1(\bp,\mathcal{I}_Y(\Delta_1))+1.
$$
From $h^1(\bp,\mathcal{I}_W(\Delta_1))\neq 0$ and $h^1(\bp,\mathcal{I}_Y(\Delta_1))=1$, we have $$h^0(\bp,\mathcal{I}_W(\Delta_1))>h^0(\bp,\mathcal{I}_Y(\Delta_1)).$$ But this contradicts with the Cayley-Bacharach property of $Y$ with respect to $\mathcal{O}_\bp(\Delta_1)$ (\cite[Theorem~1.1]{K}). Therefore, we conclude that $X=Y$, i.e., $X$ is the complete intersection of hypersurfaces of degrees $d_1,d_2,\dotsc,d_e$.

Suppose that the statement holds for $a-1$ case, where $a\geq 2$. For any integer $i\geq 2$, we consider the following exact sequence deduced from (\ref{3.2}),
\begin{equation*}
H^{i-1}(\bp,(\mathcal{I}^{a-1}_X/\mathcal{I}^a_X)(\Delta_{a}+1-i))\rightarrow H^{i}(\bp,\mathcal{I}^a_X(\Delta_a+1-i))\rightarrow H^{i}(\bp,\mathcal{I}^{a-1}_X(\Delta_a+1-i)).
\end{equation*}
It follows from Theorem \ref{mainthm} (i) that $H^{i}(\bp,\mathcal{I}^{a-1}_X(\Delta_a+1-i))=0$ for any integer $i\geq 2$. Meanwhile, we see that $H^{i-1}(\bp,(\mathcal{I}^{a-1}_X/\mathcal{I}^a_X)(\Delta_{a}+1-i))=0$ for any integer $i\geq 2$ due to the assumption $\dim X=0$. Thus, we obtain $H^{i}(\bp,\mathcal{I}^a_X(\Delta_a+1-i))=0$ for any integer $i\geq 2$. Since $\mathcal{I}^a_X$ is not $(\Delta_a+1)$-regular, then we have $H^1(\bp,\mathcal{I}^a_X(\Delta_a))\neq 0$. Moreover, we can prove the following.

\begin{claim}\label{c3.1}
	$H^1(\bp,\mathcal{I}^{a-1}_X(\Delta_{a-1}))\neq 0$.
\end{claim}

\textnormal{Proof of Claim \ref{c3.1}.} Similarly as $\dim X\geq 1$ case, $\mathcal{I}_X(d_1)$ is globally generated. Thus, we have the following exact sequences
\begin{equation}\label{K_a}
0\rightarrow\mathcal{K}_a\rightarrow H^0(\bp,\mathcal{I}_X(d_1))\otimes \mathcal{I}^{a-1}_X(\Delta_{a-1})\rightarrow\mathcal{I}^a_X(\Delta_a)\rightarrow 0
\end{equation}
and
\begin{equation}\label{K'_a}
0\rightarrow\mathcal{K}'_{a}\rightarrow H^0(\bp,\mathcal{I}_X(d_1))\otimes \mathcal{I}^{a-2}_X(\Delta_{a-1})\rightarrow\mathcal{I}^{a-1}_X(\Delta_a)\rightarrow 0
\end{equation}
where $\mathcal{K}_a$ is the kernel of the canonical sheaf morphism 
$$ 
H^0(\bp,\mathcal{I}_X(d_1))\otimes \mathcal{I}^{a-1}_X(\Delta_{a-1})\rightarrow\mathcal{I}^a_X(\Delta_a)
$$ 
and $\mathcal{K}'_{a}$ is the kernel of the canonical sheaf morphism
$$
H^0(\bp,\mathcal{I}_X(d_1))\otimes \mathcal{I}^{a-2}_X(\Delta_{a-1})\rightarrow\mathcal{I}^{a-1}_X(\Delta_a).
$$
Based on (\ref{K_a}), (\ref{K'_a}) and the snake lemma, we have the following commutative diagram where all rows and columns are exact
\begin{equation*}
\begin{tikzcd}
& 0 \arrow[d]                                         & 0 \arrow[d]           & 0 \arrow[d]                                           &   \\
0 \arrow[r] & \mathcal{K}_a \arrow[d] \arrow[r]                   &  H^0(\bp,\mathcal{I}_X(d_1))\otimes \mathcal{I}^{a-1}_X(\Delta_{a-1}) \arrow[r] \arrow[d] & \mathcal{I}^a_X(\Delta_a) \arrow[r] \arrow[d]         & 0 \\
0 \arrow[r] & \mathcal{K}'_{a} \arrow[d] \arrow[r]               &  H^0(\bp,\mathcal{I}_X(d_1))\otimes \mathcal{I}^{a-2}_X(\Delta_{a-1}) \arrow[r] \arrow[d] & \mathcal{I}^{a-1}_X(\Delta_{a}) \arrow[r] \arrow[d] & 0 \\
0 \arrow[r] & \mathcal{K}'_{a}/\mathcal{K}_a \arrow[r] \arrow[d] & H^0(\bp,\mathcal{I}_X(d_1))\otimes (\mathcal{I}^{a-2}_X/\mathcal{I}^{a-1}_X)(\Delta_{a-1}) \arrow[r] \arrow[d] & (\mathcal{I}^{a-1}_X/\mathcal{I}^{a}_X)(\Delta_{a})  \arrow[r] \arrow[d]                                 & 0 \\
& 0                                                   & 0                     & 0                                                     &  
\end{tikzcd}
\end{equation*}
It follows from the middle row of the diagram and Theorem \ref{mainthm} (i) that $H^2(\bp,\mathcal{K}'_{a})=0$. From $\dim\,\textnormal{Supp}\,\mathcal{K}'_{a}/\mathcal{K}_a=0$, we see that $H^1(\bp,\mathcal{K}'_{a}/\mathcal{K}_{a})=0$. Thus, we obtain $H^2(\bp,\mathcal{K}_a)=0$ by the left column of the diagram. Then we have the following exact sequence by the top row of the diagram
$$
H^0(\bp,\mathcal{I}_X(d_1))\otimes H^1(\bp,\mathcal{I}^{a-1}_X(\Delta_{a-1}))\rightarrow H^1(\bp,\mathcal{I}^a_X(\Delta_a))\rightarrow H^2(\bp,\mathcal{K}_a)=0
$$
Therefore, it follows from $H^1(\bp,\mathcal{I}^a_X(\Delta_a))\neq 0$ that  $H^1(\bp,\mathcal{I}^{a-1}_X(\Delta_{a-1}))\neq 0$. This completes the proof of Claim \ref{c3.1}.  

It follows from Claim \ref{c3.1} that $\mathcal{I}^{a-1}_X$ is not $(\Delta_{a-1}+1)$-regular. Then by the induction hypothesis, $X$ is the complete intersection of hypersurfaces of degrees $d_1,d_2,\dotsc,d_e$. This completes of the proof of Theorem \ref{mainthm} (ii).
\begin{remark}
	Another way to show that the ``if" part of Theorem \ref{mainthm} (ii) holds is to use the graded minimal free resolution of $\mathcal{I}^a_X$ (e.g. see \cite[Theorem~2.1]{GT}).
\end{remark}
\begin{remark}
	The argument of the proof of $\dim X=0$ case also works for $\dim X\geq 1$ case after making necessary changes.
\end{remark}
\begin{remark}
	By (almost) the same argument of the proof of Theorem \ref{mainthm} (ii), one can easily see that the bound of Castelnuovo-Mumford regularity estabilished by Niu (\cite[Corollary~4]{Niu}) is sharp exactly for complete intersections.
\end{remark}
\subsection*{Acknowledgement}
The author wishes to thank his advisor Lawrence Ein for helpful discussions and warm encouragement. The author is also grateful for the anonymous referee’s input which improves the exposition of the note.
%%%%%%%%%%%%%%%%%%%%%%%%%%%%%%%%%%%%%%%%%%%%%%%%%
\bibliographystyle{abbrv}
\bibliography{reference}

\end{document}